\documentclass[12pt]{amsart}
\usepackage{a4wide}
\usepackage{amsmath,amssymb}
\newtheorem{thm}{Theorem}[section]
 \newtheorem{cor}[thm]{Corollary}
 \newtheorem{lem}[thm]{Lemma}
 \newtheorem{prop}[thm]{Proposition}
 \theoremstyle{definition}

\newtheorem{ex}[thm]{Example}
\newtheorem{exs}[thm]{Examples}
 \theoremstyle{remark}
 \newtheorem{rem}[thm]{Remark}
 \numberwithin{equation}{section}

\DeclareMathOperator{\Z}{\mathbb{Z}}
\DeclareMathOperator{\F}{\mathbb{F}}

\begin{document}
 \date{\today}

\author[   Ko\c{s}an, Leroy, Matczuk]{    M. Tamer  Ko\c{s}an,  Andr\'e Leroy,  Jerzy Matczuk
\address{Department of Mathematics, Gebze Technical University,\\ 41400 Gebze/Kocaeli,  Turkey}
\email{mtkosan@gtu.edu.tr;  tkosan@gmail.com }
\address{Department of Mathematics, Universit\'e d'Artois, Rue Jean Souvraz, 62307 Lens, France}
\email{andre.leroy@univ-artois.fr}
\address{Department of Mathematics, University of Warsaw, Ul. Banacha, 2, Warszawa, Poland}
\email{jmatczuk@mimuw.edu.pl}}
\title{On $UJ$-rings}
\keywords{Units,  Jacobson radical, clean rings}  \subjclass[2000]{16N20, 16U60}

 \thanks{ This work was done while the second author was visiting 
 the Gebze Technical University and the  University of Warsaw.  He would like to thank these institutions for their warm welcome and TUBITAK of Turkey for the financial support.}

\maketitle

\begin{abstract}

 $UJ$-rings are studied,  i.e.  ring in which all units can be presented in a form $1+x$, for some $x\in J(R)$.  The behavior of $UJ$-rings under various algebraic construction is investigated. In particular, it is  shown that the problem of lifting   the $UJ$ property from a ring $R$ to the polynomial ring $R[x]$ is equivalent to the K\"{o}the's problem for $\F_2$-algebras.

 \end{abstract}

\section*{Introduction}

Throughout the paper all rings considered are associative and unital,   except Section 2 where nil rings naturally appear.
For a ring $R$, the Jacobson radical, the group of units and the set of all nilpotent elements  of $R$ are denoted by $J(R)$, $U(R)$ and $N(R)$, respectively.

The influence on the structure of  rings of properties  defined elementwise is intensively studied in the literature.  For example,  clean rings and their generalizations, rings with special types of units, generalizations of commutative   rings  have been investigated in relation to various global ring properties.

Let us notice that $1+J(R)$ is always contained in $U(R)$, the aim of the paper is to investigate rings in which the equality $U(R)=1 + J(R)$ holds.
A ring $R$ with this property will be called a $UJ$-ring, we will alternatively say that  $R$ has the $UJ$ property.
We  will mainly  focus on the behavior of $UJ$ property under some classical ring constructions.

Recall that $UU$-rings,  defined as rings with $U(R)=1+N(R)$ (i.e. rings with unipotent units) were introduced by C\v{a}lug\v{a}reanu    \cite{Cal} and  studied in details by Danchev and Lam in \cite{DL}.   Of course when $R$ is a $UJ$-ring with nil Jacobson radical, then $R$ is a $UU$-ring.

 Section 1  provides examples and gives some characterizations and basic properties of $UJ$-rings.

The behavior of $UJ$ property under some classical ring constructions is studied in Section 2. In particular,    it is proved (cf. Proposition \ref{necessary condition for R[x] to be UJ}) that if the polynomial ring  $R[x]$ has the $UJ$ property then $R$ is a $UJ$-ring and the Jacobson radical $J(R)$ is nil. Moreover, as Theorem \ref{koethe} shows,  the converse of the above statement is equivalent to the K\"othe's problem  for $\mathbb F_2$-algebras. Theorem \ref{morita} offers a description of Morita contexts which are $UJ$-rings.

The last section is devoted to study of some relations between $UJ$-rings and clean rings. In particular some    characterizations of clean $UJ$-rings are presented.

\section{Preliminaries}

A ring $R$ is said to be a $UJ$-ring if $1+J(R)=U(R)$.   Since units lift modulo
the Jacobson radical, $R$ is a $UJ$-ring if and only if the factor ring $R/J(R)$  is a ring with
trivial units, i.e.  $U(R/J(R))=\{1\}$.

Recall  that $J(R)$ is the largest ideal of $R$ consisting of quasi-regular elements of $R$, i.e. invertible elements in the circle monoid $(R,\circ)$  (see    \cite[Exercises for  \S 4 ]{L}). In the following lemma $\mathcal{C}(R)$ denotes the set of all quasi-regular elements of $R$.
 $(\mathcal{C}(R), \circ )$ is a group isomorphic to $U(R)$ and the isomorphism is given by $\mathcal{C}(R)\ni x \leftrightarrow 1-x\in U(R)$.  Therefore  $R$ is a $UJ$-ring if and only if $\mathcal{C}(R)$ is an ideal of $R$. This description can be used as a definition of $UJ$-rings for rings without unity.

 In the following lemma we collect other characterizations of $UJ$-rings.

\begin{lem} \label{lem char UJ} For a ring $R$,   the following conditions are equivalent:
 \begin{enumerate}
   \item  $U(R)=1+J(R)$, i.e. $R$ is a $UJ$-ring;
   \item $U(R/J(R))=\{1\}$;
   \item $\mathcal{C}(R)$ is an ideal of $R$ (then   $\mathcal{C}(R)=J(R)$);
   \item $rb-cr\in J(R)$, for any $r\in R$ and $b,c\in \mathcal{C}(R)$;
   \item $ru-vr\in J(R)$, for any  $u,v\in U(R)$ and $r\in R$;
   \item $U(R)+ U(R)\subseteq J(R)$ (then U(R)+U(R)=J(R)).

 \end{enumerate}

\begin{proof} The equivalence of $(1)- (3)$ was already observed.   The implication $(3)\Rightarrow (4) $ is trivial.

  Setting $b=1+u, c=1+v$, for $u,v\in U(R)$, and applying $(4)$ we get (5).

        Taking $r=1$ in (5)   we get $u-v\in  J(R)$, for any $u,v\in U(R)$ and $U(R)+ U(R)\subseteq J(R)$ follows. Notice that every $a\in J(R)$ can be written as a sum of two units: $a=1+(a-1)$, so (6) holds.

   Finally, using (6) we get  $U(R)- 1\subseteq  J(R)$, i.e. (1) holds.
\end{proof}
\end{lem}
 Let us now mention a few basic  examples of $UJ$-rings.

\begin{exs}
\begin{enumerate}
\item Any ring with trivial units is $UJ$. In particular, the class of $UJ$-rings contains: all Boolean rings,  all free, both commutative and noncommutative, algebras  over the field $\F_2$.
\item Any local ring $R$ with a maximal ideal $M$ such that $R/M=\mathbb{F}_2$. In particular the rings
$\Z/2^n\Z,\; \Z_{(2\Z)}$ and $ R=\mathbb F_2[[x]]$ are $UJ$.
\item If $R$ is a $UJ$-ring, then ring $T_n(R)$ of $n$ by $n $ upper triangular matrices over  $R$ and $R[x]/(x^n)$  are $UJ$-rings.
\end{enumerate}
\end{exs}
In the following proposition, we collect some basic properties of $UJ$-rings.

\begin{prop} \label{basic}
Let  $R$ be a $UJ$-ring. Then:
\begin{enumerate}

\item  $2\in J(R)$;
\item If $R$ is a division ring, then $R=\mathbb{F}_2$;
\item  $R/J(R)$ is reduced (i.e. it has no nonzero nilpotent elements) and hence abelian (i.e. every idempotent is central);
\item  If $x,y\in R$ are such that $xy\in J(R)$, then $yx\in J(R)$ and  $xRy, yRx\subseteq J(R)$;
\item Let $I\subseteq J(R)$ be an  ideal of $R$.  Then $R$ is a
$UJ$-ring if and only if $R/I$ is a $UJ$-ring;
\item $R$ is Dedekind finite;
\item The ring $\prod_{i\in I} R_i$ is $UJ$ if and only  rings $R_i$ are $UJ$, for all $i\in I$.
\end{enumerate}
\end{prop}
\begin{proof}
Statements (1) and (2),(3) are   direct  consequences of  Lemma \ref{lem char UJ}(6) and  Lemma \ref{lem char UJ}(2), respectively.  (4) follows from (3).

If $I\subseteq J(R)$, then $(R/I)/J(R/I)\simeq R/J(R)$. This gives (5).

  By (3) $R/J(R)$ is reduced, so it is Dedekind finite. Let  $a,b\in R$ be such that $ab=1$. Then, as $R/J(R)$ is Dedekind finite, we get $ba-1\in J(R)$. Thus  the idempotent $ba$ is invertible, so $ba=1$ and (6) follows.

The last statement is a consequence of the facts that $J(\prod_{i\in I} R_i)=\prod_{i\in I} J(R_i)$ and $U(\prod_{i\in I} R_i)=\prod_{i\in I} U(R_i)$
\end{proof}

 Statements (5), (3), (2)   of the above proposition    give immediately the following  characterization of semilocal $UJ$-rings.

  \begin{prop} A semilocal ring $R$ is $UJ$ if  and only if  $R/J(R) \simeq \mathbb{F}_2\times \ldots \times \mathbb{F}_2$.
  \end{prop}
In particular we have
  \begin{cor} The ring $\mathbb{Z}_n=\Z/n\Z$ is $UJ$ if and only if  $n$ is a power of 2.
  \end{cor}
  Let us finish this section with the following:
  \begin{rem}\label{$UJ$ nil elements}
  A ring $R$ is a $UJ$-ring with nil Jacobson radical  if and only if $R$ is a  $UU$-ring and $N(R)$ is an ideal of $R$.
   \end{rem}

 The following  example of Bergman (see  \cite[Example 2.5]{DL})  shows that   $UU$-rings with nil Jacobson radical do not have to be $UJ$.
 \begin{ex}
 Let $R$ be the
$\mathbb{F}_2$-algebra generated by $x,y$ with the only relation $x^2=0$.  Then   $U(R)=1+ \mathbb{F}_2x + xRx$, so $R$  is a $UU$-ring. Moreover $J(R)=0$, so $R$ is not a $UJ$-ring.
 \end{ex}

  \section{UJ property under algebraic constructions}
  The main purpose of this section is to clarify  the  connection between   K\"{o}the's problem  and $UJ$ property of rings. Later on  we present necessary and sufficient conditions for a Morita context   to be a $UJ$-ring.

 It is known and easy to check  (see \cite{DL})  that  a subring of a $UU$-ring is always a $UU$-ring. We will see in the example below that the  $UJ$ property is not hereditary on subrings but anyway we have the following:

\begin{prop} Let   $R$ be a $UJ$-ring and $Z$ a subring of $R$ such that $U(Z)=U(R) \cap Z$.  Then   $Z$ is also a $UJ$-ring. In particular this applies to  $Z=Z(R)$  the center of $R$.
 \begin{proof} Since $U(Z)=U(R)\cap Z$, we also have $ J(R)\cap Z\subseteq J(Z)$. Thus, using
  $U(R)=1+J(R)$ we get     $1+J(Z)\subseteq U(Z) = U(R) \cap Z= ( 1+ J(R) ) \cap Z = 1+ (J(R) \cap Z)\subseteq 1+ J(Z)$ and $U(Z)=1+J(Z)$ follows.
 \end{proof}
\end{prop}

 \begin{ex}
  Let $R=\mathbb{F}_2[[x]]$. Then $R$ is a $UJ$-ring  and its subring $S$   generated by  $1+x$ and
$(1+x)^{-1}=\sum_{i=0}^\infty x^i$ is not a $UJ$-ring, as it is isomorphic to
$\mathbb{F}_2[x,x^{-1}]$.
\end{ex}
Now we will concentrate on the $UJ$ property of polynomial rings.  In this context   let us notice that:
\begin{lem}\label{trivial units}
Let  $R$ be a ring with trivial units. Then $U(R[X])=\{1\}$, where    $R[X]$ denotes the polynomial ring in the  set $X$ of commuting indeterminates.
\begin{proof} Since being a unit in $R[X]$ is a local property, i.e. depends only on finitely many indeterminates, we may assume that $X$ is a finite set.

By assumption $U(R)=\{1\}$, so $R$ does not contain nontrivial nilpotent elements, i.e. it is a reduced ring.   \cite[Corollary 1.7]{KLM} characterizes reduced rings as rings such that $U(R[x])=U(R)$ and the thesis follows easily.
\end{proof}
\end{lem}

Let us recall that a ring $R$ is 2-primal if its prime radical $B(R)$ coincides with the set of all its nilpotent elements.

\begin{prop} \label{dort}  Let $R$ be  a 2-primal $UU$-ring. Then, for any set $X$ of commuting indeterminates, the polynomial ring  $R[X]$ is a  $UJ$-ring.
\end{prop}
\begin{proof}  It is known that $B(R[X])=B(R)[X]$ (cf.\cite[Theorem 10.19]{L}).  Thus the assumptions imposed on $R$ and Lemma \ref{trivial units} imply that the ring  $R[X]/B(R[X])\simeq (R/B(R))[X]$ has trivial units. Now, by Proposition \ref{basic}(5), $R[X]$ is a $UJ$-ring.
\end{proof}

\begin{prop}
\label{necessary condition for R[x] to be UJ}
If the polynomial ring $R[x]$ is $UJ$, then $R$ is a
$UJ$-ring and $J(R)$ is a nil ideal of $R$.
\end{prop}
\begin{proof}
 It is  known that $J(R[x])=I[x]$ for some nil ideal $I$ of $R$.  Thus, as $R[x]$ is $UJ$,  we have $1+J(R)\subseteq U(R[x])=1 +J(R[x])=1+I[x]$.  This implies that $J(R)=I$ is nil.  Then, as $R[x]$ is a $UJ$-ring,
  $\{1\}=U(R[x]/J(R[x]))=U((R/J(R)[X])$.  Hence also   $U(R/J(R))=\{1\}$, i.e. $R$ is a $UJ$-ring.
\end{proof}

The above proposition shows that if the $UJ$ property lifts from a ring $R$ to the polynomial ring $R[x]$, then $J(R)$ has to be a nil ideal.
The next theorem says that the problem of lifting the $UJ$ property is equivalent to K\"{o}the's problem for algebras over the field $\F_2$.
Recall that K\"{o}ethe's problem  (formulated in 1930)  asks whether a ring $R$    has no nonzero nil one-sided ideals provided $R$ has no nonzero nil ideals.
It is   known (see Theorem 6, \cite{K}) that the
problem has a positive solution if and only if it has positive solution for algebras
over fields.      There are many other
problems in ring theory which are equivalent or related to  it (see \cite{P}), one more is indicated below.
\begin{thm}\label{koethe}
The following conditions are equivalent:
\begin{enumerate}
\item  For any $UJ$-ring $R$ with nil Jacobson radical, the polynomial ring $R[x]$ is also  $UJ$;
\item  For any nil $\mathbb{F}_2$-algebra $A$,  $J(A[x])=A[x]$;
\item For any nil $\mathbb{F}_2$-algebra $A$ and $n\geq 1$ the matrix algebra $M_n(A)$ is nil;
\item  K\"{o}the's problem has a positive solution in the class of $\mathbb{F}_2$-algebras.
\end{enumerate}
\end{thm}
\begin{proof} The equivalence of statements (2)-(4) is a well known result of Krempa (\cite{K}).

(1) $\Rightarrow$ (2) Assume that (1) holds and    $A$ is a nil $\F_2$-algebra. Let $A^*$  be  the $\mathbb{F}_2$-algebra  obtained from $A$ by adjoining unity with the help of $\mathbb{F}_2$.  Note that  $A^* = A \cup (1 + A)$,
$J(A^*)= A$ and $A^*/J(A^*)=\mathbb{F}_2$.  In particular,   $A^*$ is a $UJ$-ring and, by (1),  $A^*[x]$ also has the $UJ$ property. Consequently, $U(A^*[x]/J(A^*[x]) )=\{1\}$ follows.  Let $N$ be the ideal of $A^*$ such that  $J(A^*[x])=N[x]$. As $U((A^*/N)[x])=\{1\}$,  we get $A^*/N$ is reduced.  This yields $A=N$, i.e $J(A[x])=A[x]$.

(2)  $\Rightarrow$ (1) Let $R$ be a ring as in (1). By Proposition \ref{basic}(1),   $2\in J(R)$.   Thus, as $J(R)$ is nil,  $(2R)[x]$ is a nilpotent ideal of $R[x]$. Therefore, in virtue of  Proposition \ref{basic}(5), to show  that $R[x]$ has the $UJ$ property, it is enough to prove that  $R[x]/(2R[x])\simeq (R/2R)[x]$ is a  $UJ$-ring.  Thus,  eventually replacing $R$ by $R/2R$,  we may assume, that $2=0$ in $R$, i.e. $R$ is an algebra over  the field $\mathbb{F}_2$. Then,  the property (2) gives $J(R[x])=J(R)[x]$ (because  $J(R)$ is nil), and  so $R[x]/J(R[x])\simeq (R/J(R))[x]$. Since $R$ is $UJ$, we get  $U(R/J(R))=\{1\}$ and Lemma \ref{trivial units} implies that $U((R/J(R))[x])  =\{1\}$. This proves that $R[x]$ is $UJ$, as desired.
\end{proof}

Let us observe that whenever  $n>1$,  the matrix ring $M_n(R)$ does not have the  $UJ$ property.   Indeed, the ring $M_n(R)/J(M_n(R))\simeq M_n(R/J(R))$ is not reduced when $n>1$, so   $M_n(R)$ can not be  $UJ$, as observed in  Proposition \ref{basic}(3).
\begin{prop} \label{thm corners}  Let $R$ be a ring   with an  idempotent $e\in R$. The following conditions are equivalent:
\begin{enumerate}
  \item $R$ is a $UJ$-ring;
  \item $eRe$ and $(1-e)R(1-e)$  are $UJ$-rings, and   $eR(1-e),(1-e)Re\subseteq J(R)$.
\end{enumerate}
\end{prop}
 \begin{proof}
 Suppose $R$ is a $UJ$-ring. Then, taking $x=e$ and $y=1-e$ in Proposition \ref{basic}(4),  we obtain $eR(1-e),(1-e)Re\subseteq J(R)$. Recall that $J(eRe)=J(R)\cap eRe$, thus the natural homomorphism form $eRe$ into $R/J(R)$ induces an embedding of $eRe/J(eRe)$ into $R/J(R)$.  Moreover, by Proposition \ref{basic}(3), $\bar e=e+J(R)$ is a central idempotent of $\bar R=R/J(R)$. Thus  $\{\bar 1\}=U(\bar R)=U(\bar e \bar R)\times
 U((\bar 1-\bar e)\bar R)$, so the ring $eRe/J(eRe)\simeq \bar eR$ has trivial units, i.e. $eRe$ is a $UJ$-ring. Similarly, $(1-e)R(1-e)$ is a $UJ$-ring.

 Suppose (2) holds. Making use of Pierce decomposition of $R$ with respect to $e$ and the assumption that $eR(1-e),(1-e)Re\subseteq J(R)$, it is clear that $R/J(R)\simeq eRe/J(eRe)\times (1-e)R(1-e)/J((1-e)R(1-e))$ and $U(R/J(R))=\{\bar 1\}$  follows as both $eRe$ and $(1-e)R(1-e)$ are $UJ$-rings.
  \end{proof}

  The above proposition can be extended to Morita context but instead of   using Proposition \ref{basic} we will use  the description of $N$-radicals (Jacobson radical is such) of Morita contexts given in  \cite[Theorem \ref{thm corners}]{GW}.

  Let us recall that a quadruple  $(R, V, W, S)$ is a Morita context
where $R,S$ are rings, $V$, $W$ are  $(R-S)$ and  $(S-R)$ bimodules, respectively and the   products $\phi : V\otimes _SW\rightarrow  R$ and $\psi :W\otimes _RV\rightarrow S$ are given such that matrices $\left(
            \begin{array}{cc}
              R & V \\
              W & S \\
            \end{array}
          \right)
$  form an associative ring with natural matrix operations defined with the help of $\phi$ and $\psi$.

\begin{thm}\label{morita}
 Let $(R, V, W, S)$ be  a Morita context and $T:=\left(
            \begin{array}{cc}
              R & V \\
              W & S \\
            \end{array}
          \right)
$. The following conditions are equivalent:
 \begin{enumerate}
   \item  $T$ is a $UJ$-ring;
   \item  $R$, $S$ are $UJ$-rings and $VW\subseteq J(R)$, $WV\subseteq J(S)$;
   \item $R$, $S$ are $UJ$-rings and $T/J(T)\simeq R/J(R)\oplus S/J(S)$.
 \end{enumerate}

\end{thm}
\begin{proof}
 By  \cite[Theorem 3.18.14]{GW}, we have $J(T)=\left(
            \begin{array}{cc}
              J(R) & B \\
              C & J(S) \\
            \end{array}
          \right)$, where $B=\{v\in V\mid Wv\subseteq J(S)\}=\{v\in V\mid vW\subseteq J(R)\}$ and $C=\{w\in W\mid wV\subseteq J(S)\}=\{w\in W\mid Vw\subseteq J(R)\}$.

    $(1)\Rightarrow (2)$      Suppose that $T$ is a $UJ$-ring. Then  $T/J(T)$ does not possess   nonzero nilpotent elements.  This forces $ \left(\begin{array}{cc}
                                                                                                                          0 & V \\
                                                                                                                         0 & 0
                                                                                                                        \end{array}\right), \left(\begin{array}{cc}
                                                                                                                          0 & 0 \\
                                                                                                                         W & 0
                                                                                                                        \end{array}\right)
           \subseteq J(T)$, i.e. $B=V$,  $C=W$, $VW\subseteq J(R)$, $WV\subseteq J(S)$ and $T/J(T)\simeq  R/J(R)\oplus S/J(R)$. Thus, by (3) and (7) of Proposition \ref{basic}, $R,S$ are $UJ$-rings, i.e. (2) holds.

           Implications $(2)\Rightarrow (3)$  and $(3)\Rightarrow (1)$ are    consequences of  \cite[Theorem 3.18.14]{GW} and Lemma \ref{lem char UJ}, respectively.
\end{proof}

\section{Clean rings and $UJ$ property}

Recall that an element $r\in  R$ is clean ($J$-clean)
provided   there exist an idempotent $e\in R$ and an element $t\in U(R)$ ($t\in J(R)$) such that
$r = e + t $.  A ring $R$ is clean ($J$-clean) if every element of $R$ has such clean ($J$-clean) decomposition. It is known that every $J$-clean ring is clean (in fact if $-r=e+j$ is a $J$-clean decomposition of $-r\in R$, then $r=(1-e)+ (-1-j)$ is a clean decomposition of $r$).

\begin{prop}\label{thm:jc} For a ring $R$, the following conditions are equivalent:
\begin{enumerate}
\item $R$ is a $UJ$-ring.
\item All  clean elements of $R$ are J-clean.
\end{enumerate}
\end{prop}
\begin{proof} $(1)\Rightarrow (2)$ Assume that  $r\in  R$ is a clean element and $r=e+u$ is its clean decomposition.  As $R$ is a $UJ$-ring, $2\in J(R)$ and $u=1+j$ for some $j\in J(R)$. Then $2e+j\in J(R)$ and  $r=e+1+j=(1-e) + (2e+j)$ is a $J$-clean decomposition of $r$, i.e. (2) holds.

 $(2) \Rightarrow (1)$ Let $u\in U(R)$. Then $u$ is a  clean element and, by the
 hypothesis,   $u$ is  $J$-clean. Let    $u=e+j$ be a $J$-clean decomposition of $u$.
 Since $1=eu^{-1}+ju^{-1}$, we obtain that   $eu^{-1}=1-ju^{-1}$ is a unit of $R$.
 Hence   $e=1$. This means  that   $u=1+j$ and $U(R)=1+J(R)$ follows.
\end{proof}

\begin{thm} \label{thm clean}
For a ring $R$, the following conditions are equivalent:
 \begin{enumerate}
   \item   $R$ is a clean   $UJ$-ring;
   \item  $R/J(R)$ is a Boolean ring and idempotents lift modulo $J(R)$;
   \item $R$ is a $J$-clean, $UJ$-ring;
   \item $R$ is a $J$-clean ring.
 \end{enumerate}
\end{thm}
\begin{proof} $(1)\Rightarrow (2)$
 The imposed assumptions imply that $R/J(R)$ is a clean ring such that $U(R)=\{1\}$. In particular, $2=0$ in $(R/J(R))$ and every element $r\in R/J(R)$ is of the form $r=e_r+1$, for a suitable idempotent $e_r$. Hence $r^2=r$, i.e. $R/J(R)$ is Boolean. It is known (cf.\cite[Lemma 17]{NZ}), that  idempotents lift modulo every
ideal $I$ of a clean ring $R$, so (2) follows.

 $(2)\Rightarrow (3)$ Suppose (2) holds and let $a\in R$. Then $a+J(R)\in R/J(R)$ is an idempotent. Hence there exists an idempotent $e\in R$ such that $a-e\in J(R)$, i.e. $a$ is a $J$-clean element. This shows that $R$ is $J$-clean. If $u\in U(R)$, then $u+J(R)$ is a unit in a Boolean ring $  R/J(R)$. Thus $u-1\in J(R)$, so $R$ is a $UJ$-ring.

 $(3)\Rightarrow (4)$ This is a tautology.

 $(4)\Rightarrow (1)$ This implication is given by  Proposition \ref{thm:jc}.
\end{proof}

It is known (cf.  \cite[Theorem 5.9]{Di}) that a ring $R$ is uniquely nil clean if and only if  $R/J(R)$ is Boolean, $J(R)$ is nil and idempotents lift uniquely modulo $J(R)$.  In particular,  the class of  uniquely nil clean ring is contained in the class of $UJ$-rings. However, slightly bigger class of conjugate nil clean rings is not included in $UJ$-rings, as the ring $M_2(\mathbb{F}_2)$ is conjugate nil clean (see  \cite[Corollary 2.4]{M}) but it is not a $UJ$-ring.   Assuming additionally in Theorem \ref{thm clean} that $J(R) $ is a nil ideal, we get:

\begin{thm}  For a ring $R$, the following conditions are equivalent:
\begin{enumerate}\label{thm nil clean}
\item $R$ is a clean $UJ$-ring with nil Jacobson radical $J(R)$;
\item $R/J(R)$ is a Boolean ring and $J(R)$ is nil;
    \item  $R$ is a nil clean $UJ$-ring;
        \item $R$ is a conjugate nil clean $UJ$-ring;

  \item $R$ is a conjugate nil clean ring and $N(R)$ is an ideal of $R$;
   \item $R/J(R) $ is a Boolean ring and $R$ is a $UU$-ring.
\end{enumerate}
\end{thm}
\begin{proof}
By   \cite[Corollary 3.17]{Di},  $R$ is a nil clean ring  if and only if $R/J(R) $ is nil clean and $J(R)$ is  nil. In particular, when $R$ is a $UJ$-ring, then  $R$ is nil clean   if and only if $R$ is  $J$-clean   and  $J(R)$ is a nil ideal of $R$.  Now the equivalence of statements $(1)- (3)$ is given by Theorem \ref{thm clean} and the fact that idempotents lift modulo nil ideals.

The implication $(4)\Rightarrow (3)$ is a tautology.

Statement $(2)$ implies that $R$ is a $UJ$-ring, thus the implication $(2)\Rightarrow (4)$ is a consequence of  \cite[Corollary 2.16]{M} and the fact that Boolean rings  are conjugate nil clean.

If $R$ is  nil clean, then $J(R)$ is nil. Therefore, the equivalence $(4)\Leftrightarrow (5)$ is a consequence of Remark \ref{$UJ$ nil elements}.

Finally, one can easily check that both (2) and (6) are equivalent to $R/J(R)$ is Boolean and $J(R)=N(R)$.
\end{proof}

Comparing Theorems \ref{thm clean} and \ref{thm nil clean}, let us observe that the class of $J$-clean rings having nil Jacobson  radical is equal to the class of $UJ$ nil clean rings but it is strictly contained in the class of all nil clean rings, as the  the ring $M_2(\F_2)$ is nil clean however it is     not a  $UJ$-ring.

\end{document}